\documentclass[12pt]{amsart}

\usepackage{mathtools,amsthm,amssymb,tikz-cd}
\usepackage[colorlinks,linkcolor=red!75!black,citecolor=green!50!black,urlcolor=cyan!75!black]{hyperref}

\usepackage{lmodern}
\usepackage{microtype}

\usepackage[margin=1in]{geometry}
\renewcommand{\footnoterule}{\vfill\kern -3pt \hrule width 0.4\columnwidth \kern 2.6pt}

\usepackage{titlesec}

\titleformat{\section}{\center\normalsize\bfseries}{\thesection. }{0ex}{}

\titleformat{\subsection}{\normalsize\bfseries}{\hspace{3em}\thesubsection. }{0pt}{}

\titleformat{\subsubsection}[runin]{\normalsize\bfseries}{\thesubsection}{0.5em}{}[.]

\usepackage{enumitem}
\setlist{topsep=0.5ex,itemsep=0ex,parsep=0ex,leftmargin=3em,listparindent=\parindent}

\newtheoremstyle{thmstyle}{}{}{\itshape}{}{\bfseries}{. }{0pt}{\thmnumber{#2.\ }\thmname{#1}\thmnote{\textnormal{ (#3)}}}

\newtheoremstyle{defnstyle}{}{}{}{}{\bfseries}{.}{ }{\thmnumber{#2.\ }\thmname{#1}\thmnote{\textnormal{ (#3)}}}

\newtheoremstyle{thingstyle}{}{}{}{}{\bfseries}{. }{0pt}{}

\theoremstyle{thmstyle}
\newtheorem{thm}{Theorem}
\newtheorem{lem}[thm]{Lemma}

\theoremstyle{defnstyle}

\newtheorem{rmk}[thm]{Remark}

\theoremstyle{thingstyle}

\counterwithin{equation}{thm}

\newcommand{\mc}{\mathcal}
\newcommand{\mr}{\mathrm}
\newcommand{\tit}{\textit}
\newcommand{\tn}{\textnormal}

\newcommand{\ol}{\overline}

\newcommand{\F}{\mathbb{F}}
\newcommand{\Q}{\mathbb{Q}}
\newcommand{\Z}{\mathbb{Z}}

\newcommand{\ceq}{\coloneqq}
\newcommand{\cln}{\colon}

\newcommand{\onto}{\mathrel{\mathrlap{\rightarrow}\mkern-2.25mu\rightarrow}}
\newcommand{\ov}{\,|\,}

\numberwithin{thm}{section}

\newtheoremstyle{discussionstyle}{}{}{}{}{\bfseries}{. }{0pt}{\thmname{#1}\thmnumber{#2}\thmnote{\textnormal{ (#3)}}}

\theoremstyle{discussionstyle}
\newtheorem*{acknowledgements}{Acknowledgements}
\newtheorem*{ontheproof}{On the proof}
\newtheorem*{notation}{Notation}

\newcommand{\Ab}{\mathrm{ab}}

\newcommand{\cG}{\mathcal{G}}
\newcommand{\nsolv}{{n\textnormal{-}\mathrm{solv}}}
\newcommand{\solv}{\mathrm{solv}}

\newcommand{\Bibkeyhack}[3]{}

\title[Torsion over solvable extensions]{Torsion of Abelian varieties over solvable extensions of number fields}
\author{Jake Huryn}
\address{The Ohio State University}
\email{huryn.5@osu.edu}
\date{}


\begin{document}
\begin{abstract}
Let $K$ be a number field, and let $A$ be an Abelian variety over $K$ which has no CM isogeny-factors over $\overline{K}$.
We prove that $A$ has only finitely many torsion points over the maximal $n$-step-solvable extension of $K$ for any $n$ and only finitely many torsion points of prime order over the maximal prosolvable extension of $K$.
\end{abstract}

\maketitle

\section{Introduction}

\noindent
Let $A$ be an Abelian variety over a number field $K$.
In this note, we study the torsion of Abelian varieties over certain large algebraic extensions of $K$.

As an example of such a result, Ribet has shown that $A$ has only finitely many torsion points over the field obtained by adjoining all roots of unity to $K$ \cite[Appendix]{katz-lang}.

To get finiteness over larger extensions of $K$, it is often necessary to make some assumptions on the endomorphisms of $A$.
Indeed, if $A$ has CM over $K$, i.e.\ $\mr{End}_K(A)\otimes\Q$ includes a semisimple $\Q$-subalgebra of dimension $2\dim(A)$, then every torsion point of $A$ is defined over $K^\Ab$, the maximal Abelian extension of $K$.
On the other hand, Zarhin has shown that $A$ is isogenous to a product of $K$-simple Abelian varieties none of which have CM over $K$ if and only if $A$ has only finitely many torsion points over $K^\Ab$ \cite[Corollary to Theorem 1]{zarhin}.

Our main theorem sharpens Zarhin's result (with one caveat---see Remark \ref{rmk}(a)).
Let $K^\nsolv$, for $n\geq0$, be the composite of all $n$-step-solvable extensions of $K$, and let $K^\solv$ be the composite of all solvable extensions of $K$, so that $K^\solv=\bigcup_{n\geq0}K^\nsolv$.

\begin{thm}
\label{thm:main}
Let $A$ be an Abelian variety over a number field $K$.
Assume that $A_{\ol K}$ is isogenous to a product of non-CM simple Abelian varieties.
Then
\begin{enumerate}
\item
$A$ has only finitely many torsion points over $K^\nsolv$ for al{}l $n$.
\item
$A$ has only finitely many torsion points of prime order over $K^\solv$.
\end{enumerate}
\end{thm}

\begin{rmk}
\label{rmk}
\mbox{}
\begin{enumerate}
\item
Our hypothesis on $A$ is stronger than Zarhin's.
This is unavoidable:
If, for example, $E$ is an elliptic over $\Q$ with CM over $\ol\Q$, then $E$ does not have CM over $\Q$, but nonetheless all its torsion is defined over $\Q^{2\tn{-}\mathrm{solv}}=(\Q^\Ab)^\Ab$ since its CM is defined over a quadratic extension of $\Q$.
\item
There may be infinitely many torsion points of $A$ over $K^\solv$:
If the $\mr{Gal}(\ol K\ov K)$-action on the $\ell$-adic Tate module of $A$ is pro-$\ell$, e.g.\ if $A$ has all its $\ell$-torsion defined over $K$, then for \tit{al{}l} $n\geq0$, every $\ell^n$-torsion point of $A$ is defined over $K^\solv$, since any pro-$\ell$ group is prosolvable.
\item
A few months after this note was posted to the arXiv, Davide Lombardo and Tam\'as Szamuely kindly alerted the author that they had independently obtained part (a) of Theorem \ref{thm:main} as a consequence of a result that applies to general \'etale cohomology groups \cite[Corollary 1.4]{ls}; we later noticed that part (b) follows from their arguments as well \cite[Remark 3.4(2)]{ls}.
Also, as one of the referees pointed out, part (b) can be deduced from the (very general) main theorems of the paper \cite{hui}.
Nevertheless, I hope that the short proof presented here will be useful for those interested primarily in the case of torsion of Abelian varieties.
\end{enumerate}
\end{rmk}

\begin{ontheproof}
The principal input is the main theorems of \cite{zarhin}, bolstered by results of Serre, Larsen--Pink, and Wintenberger on the images of mod-$\ell$ Galois representations (explained in \cite[Proposition 2.10]{zywina}).
\end{ontheproof}

\begin{acknowledgements}
I warmly thank Tyler Genao and John Yin for introducing me to Ribet's and Zarhin's theorems.
I also thank the referees, whose comments greatly improved the quality of this document and, in particular, pointed out a mistake in an earlier version of the proof.
This material is based upon work supported by the National Science Foundation under Grant No.\ DMS-2231565.
\end{acknowledgements}

\begin{notation}
Throughout, $K$ is a fixed number field with absolute Galois group $\Gamma\ceq\mr{Gal}(\ol K\ov K)$, and $A$ is a fixed Abelian variety over $K$.
For a prime $\ell$, the group of $\ell^n$-torsion points of $A$ over $K$ is denoted by $A[\ell^n]$, and we will use the standard notation $\mr T_\ell(A)\ceq\varprojlim\!{}_\ell\,A[\ell^n]$ and $\mr V_\ell(A)\ceq\mr T_\ell(A)\otimes_{\Z_\ell}\Q_\ell$ for the $\ell$-adic Tate modules of $A$.

For a (pro)finite group $\Pi$, we denote by $\Pi'$ and $\Pi^{(n)}$, respectively, its derived and its $n$th derived subgroups (which, recall, are the \tit{closures} of the groups generated by commutators or $n$-fold commutators).
The fields $K^\nsolv$ and $K^\solv$ are as defined above.
Note that, by definition, $\mr{Gal}(\ol K\ov K^\nsolv)=\Gamma^{(n)}$ and $\mr{Gal}(\ol K\ov K^\solv)=\bigcap_n\Gamma^{(n)}$.

Finally, the proof relies on the theory of algebraic groups, and we will be very careful to always distinguish between an algebraic group (a scheme) and its group of points with values in some ring (an abstract group).
\end{notation}

\section{Proof of Theorem \ref{thm:main}}

\noindent
We are going to prove (b) plus the following statement (a$'$), which together imply (a):
\begin{enumerate}
\item[(a$'$)]
$A$ has only finitely many torsion points of $\ell$-power order over $K^\nsolv$ for all $\ell$ and $n$.
\end{enumerate}
By enlarging $K$, replacing $A$ by an isogenous Abelian variety, and arguing for the factors of $A$ individually, we may and do assume that $A$ is absolutely simple and that for \tit{every} $\ell$, the image of $\Gamma$ in $\mr{GL}(\mr V_\ell(A))$ has connected Zariski closure (i.e.\ $K_A^\mr{conn}=K$, in the notation of \cite[\textsection1]{zywina}).
The latter assumption is justified by work of Serre \cite[2.2.3]{serre-resume} (see also \cite[Proposition 6.14]{lp}).

\begin{proof}[Proof of \tn{(a$'$)}]
Fix $\ell$ and $n$.
By an easy argument using the fact that an inverse limit of nonempty sets is nonempty, it is enough to show that $\mr V_\ell(A)^{\Gamma^{(n)}}=0$.

Since $\Gamma^{(n)}$ is a normal subgroup of $\Gamma$, the $\Q_\ell$-subspace $\mr V_\ell(A)^{\Gamma^{(n)}}$ of $\mr V_\ell(A)$ is stable under the action of $\Gamma$.
Moreover, $\Gamma$ acts on $\mr V_\ell(A)^{\Gamma^{(n)}}$ through its quotient $\Gamma/\Gamma^{(n)}$, which is $n$-step-solvable.
It therefore further suffices to show that for any nonzero $\Gamma$-invariant $\Q_\ell$-subspace $W$ of $\mr V_\ell(A)$, the image of $\Gamma$ in $\mr{GL}(W)$ is not $n$-step-solvable.

Fix $W$ as above.
Let $\Gamma_W$ be the image of $\Gamma$ in $\mr{GL}(W)$, let $G_W$ be the Zariski-closure of $\Gamma_W$ in $\mr{GL}_W$, and let $H_W$ be the derived subgroup of $G_W$ (in the sense of algebraic groups).
By hypothesis, $G_W$ and $H_W$ are connected, and they are reductive by Faltings's theorem.
By \cite[Theorem 2]{zarhin}, $\Gamma_W$ is noncommutative, so $H_W$ is nontrivial.
The ``$n$-fold commutator'' map $(G_W)^{\times2^n}\to H_W$ is surjective; indeed, it is surjective when restricted to $(H_W)^{\times2^n}$, because every element of $H_W(\ol{\Q_\ell})$ is a commutator \cite{ree}.
Since $\Gamma_W$ is Zariski-dense in $G_W$, we deduce that $(\Gamma_W)^{(n)}$ is Zariski-dense in $H_W$, and in particular nontrivial.
\end{proof}

\begin{rmk}
It follows from Bogomolov's ``open image'' theorem \cite[Theorem 2]{bogomolov} and an easy argument with Lie algebras (see \cite[p.\ 1909, (1)]{cadoret-kret})  that $(\Gamma_W)^{(n)}$ is in fact an \tit{open} subgroup of $H_W(\Q_\ell)$ (given the $\ell$-adic topology) for all $n$.
\end{rmk}

\noindent
The proof of the second part of Theorem \ref{thm:main} will use the following well-known fact.

\begin{lem}
\label{lem:gt}
Let $\ell\geq5$, let $H$ be a nontrivial connected semisimple group over $\F_\ell$, let $\tilde H\to H$ be its universal cover, and let $\varphi\cln\tilde H(\F_\ell)\to H(\F_\ell)$ denote the map on $\F_\ell$-points.
Then
\begin{enumerate}
\item
$\mr{Img}(\varphi)$ is equal to $H(\F_\ell)'$, the derived subgroup of $H(\F_\ell)$.
\item
$\mr{Img}(\varphi)$ is not solvable.
\end{enumerate}
\end{lem}

\begin{proof}
Since $\tilde H$ is quasisplit by Lang's theorem, part (a) follows from \cite[Corollaire 6.5 and Remarque 6.6]{borel-tits}.
Now $\tilde H$ is a product of groups of the form $\mr{Res}_{\F_{\ell^n}|\F_\ell}(S)$ for some $n$ and simply connected and geometrically almost simple group $S$ over $\F_{\ell^n}$ \cite[Chapter X, \S1.3]{cf}.
When $\ell\geq5$, each group of $\F_\ell$-points $\mr{Res}_{\F_{\ell^n}|\F_\ell}(S)(\F_\ell)=S(\F_{\ell^n})$ is perfect and simple modulo its center---these are the Chevalley and Steinberg groups \cite[Theorem 24.17]{mt}.
So (c) holds because the kernel of $\tilde H(\F_\ell)\to H(\F_\ell)$ is central.
\end{proof}

\begin{rmk}
In the setting of Lemma \ref{lem:gt}, we in fact have $[H(\F_\ell):\mr{Img}(\varphi)]\leq2^{\mr{rank}(H)}$, as explained in \cite[Remark 3.6]{nori}.
Moreover, the kernel of $\tilde H(\F_\ell)\to H(\F_\ell)$ is the whole of the center of $\tilde H(\F_\ell)$ \cite[Corollary 24.13]{mt}.
\end{rmk}

\begin{proof}[Proof of Theorem \tn{\ref{thm:main}(b)}]
Take $\ell\gg0$.
By the same argument as in the second paragraph of the proof of \tn{(a$'$)}, it suffices to verify the following assertion:
For any nonzero $\Gamma$-stable $\F_\ell$-subspace $W_\ell$ of $A[\ell]$, the image of $\Gamma$ in $\mr{GL}(W_\ell)$ is not solvable.
We may and do further assume that $\Gamma$ acts irreducibly on $W_\ell$.

Let $\Gamma_\ell$ denote the image of $\Gamma$ in $\mr{GL}(A[\ell])$, and let $\cG$ be the Zariski-closure of the image of $\Gamma$ in $\mr{GL}_{\mr T_\ell(A)}$ (that is, the smallest closed subscheme of $\mr{GL}_{\mr T_\ell(A)}$ such that $\cG(\Z_\ell)$ includes the image of $\Gamma$).
The latter is by definition a linear algebraic group over $\Z_\ell$, and by work of Serre, Larsen--Pink, and Wintenberger (we again refer to the exposition of \cite[Proposition 2.10]{zywina}), the following holds, since $\ell\gg0$:
The special fiber $G_\ell$ of $\cG$ is a connected reductive group over $\F_\ell$, and $G_\ell(\F_\ell)'\subseteq\Gamma_\ell\subseteq G_\ell(\F_\ell)$.

Let $W_\ell$ be as above, and let $\Gamma_{W_\ell}$ denote the image of $\Gamma$ in $\mr{GL}(W_\ell)$.
To extract information about $\Gamma_{W_\ell}$ from the previous paragraph, we will show that the natural map $\Gamma_\ell\onto\Gamma_{W_\ell}$ is induced by an \tit{algebraic} quotient of $G_\ell$.
By Faltings's theorem and its mod-$\ell$ variant (see \cite[Chapter IV, \textsection4]{faltings-wustholz} or \cite[Corollary 5.4.5]{zarhin-finiteness}), the horizontal arrows of the following commutative square are isomorphisms, since $\ell\gg0$:
\[
\begin{tikzcd}
\mr{End}_K(A)\otimes_\Z\Z_\ell\arrow[r,"\sim"]\arrow[d,two heads]&
\mr{End}_\Gamma(\mr T_\ell(A))\arrow[d]\\
\mr{End}_K(A)\otimes_\Z\F_\ell\arrow[r,"\sim"]&
\mr{End}_\Gamma(A[\ell]).
\end{tikzcd}
\]
Lifting the idempotent element of $\mr{End}_\Gamma(A[\ell])$ corresponding to $W_\ell$ along the surjection $\mr{End}_K(A)\otimes_\Z\Z_\ell\onto\mr{End}_K(A)\otimes_\Z\F_\ell$,\footnote{
It is well-known (see, for example, the more general \cite[Lemma 5.4]{jannsen}) that if $R$ is a not-necessarily-commutative ring, and $I\subseteq R$ is a nilpotent two-sided ideal, then the map $R\onto R/I$ induces a surjection on idempotents.
Apply this to the maps $\mr{End}_K(A)\otimes_\Z\Z/\ell^n\onto\mr{End}_K(A)\otimes_\Z\F_\ell$; we may pass to the limit since $\mr{End}_K(A)$ is a finitely generated free $\Z$-module.
}
we in turn lift $W_\ell$ to a $\Gamma$-stable $\Z_\ell$-subspace $\mc W$ of $\mr T_\ell(A)$.
This yields a quotient map $\cG\onto\cG_{\mc W}$, the latter being the Zariski closure of the image of $\Gamma$ in $\mr{GL}_{\mc W}$.

Let $G_{W_\ell}$ denote the special fiber of $\cG_{\mc W}$ and $H_{W_\ell}$ the derived subgroup of $G_{W_\ell}$.
By \cite[Theorem 3]{zarhin}, $\Gamma_{W_\ell}$ is a noncommutative subgroup of $G_{W_\ell}(\F_\ell)$, so $H_{W_\ell}$ is nontrivial.
The quotient map $G_\ell\onto G_{W_\ell}$ (the special fiber of $\mc G\onto\mc G_{\mc W}$) restricts to a surjective map $H_\ell\onto H_{W_\ell}$; let $\psi\cln\tilde H_\ell\to\tilde H_{W_\ell}$ denote the corresponding map of universal covers, which is again surjective.
Moreover, all these groups are connected, since $G_\ell$ is so, as noted above.
Let $\mr{Ker}(\psi)^\circ$ denote the neutral component of $\mr{Ker(\psi)}$.
Then $\tilde H_\ell/{\mr{Ker}(\psi)^\circ}$ is a finite cover of $\tilde H_{W_\ell}$, so $\mr{Ker}(\psi)=\mr{Ker}(\psi)^\circ$, because $\tilde H_{W_\ell}$ is simply connected.
Thus, by Lang's theorem, the map $\tilde H_\ell(\F_\ell)\to\tilde H_{W_\ell}(\F_\ell)$ is surjective.
In summary, we have the following diagram:
\[
\begin{tikzcd}
H_\ell(\F_\ell)'\arrow[r,phantom,"\subseteq"]&
G_\ell(\F_\ell)'\arrow[r,phantom,"\subseteq"]&
\Gamma_\ell\cap H_\ell(\F_\ell)\arrow[r,phantom,"\subseteq"]\arrow[d]&
H_\ell(\F_\ell)\arrow[d]&\tilde H_\ell(\F_\ell)\arrow[l]\arrow[d,two heads]\\
&&
\Gamma_{W_\ell}\cap H_{W_\ell}(\F_\ell)\arrow[r,phantom,"\subseteq"]&
H_{W_\ell}(\F_\ell)&
\tilde H_{W_\ell}(\F_\ell).\arrow[l]
\end{tikzcd}
\]
From the diagram and Lemma \ref{lem:gt}(a) we deduce that $\mr{Img}(\tilde H_\ell(\F_\ell)\to H_\ell(\F_\ell))\subseteq\Gamma_\ell$;
therefore, by the surjectivity, $\mr{Img}(\tilde H_{W_\ell}(\F_\ell)\to\tilde H_{W_\ell}(\F_\ell))\subseteq\Gamma_{W_\ell}$;
and at last we see, from Lemma \ref{lem:gt}(b), that $\Gamma_{W_\ell}$ is not solvable.
\end{proof}

\bibliographystyle{amsalpha}
\bibliography{refs.bib}

\end{document}